\documentclass{amsart}
\usepackage{amsmath}
\usepackage{eurosym}
\usepackage{amsfonts}
\usepackage{amssymb}
\usepackage{amsmath}
\usepackage{amsfonts}
\usepackage{graphicx}
\usepackage{subfig}
\usepackage{eurosym}
\usepackage{amssymb}
\usepackage{amsmath}
\usepackage{amsfonts}
\usepackage{graphicx}
\usepackage{subfig}

\newtheorem{theorem}{Theorem}
\theoremstyle{plain}

\newtheorem{definition}{Definition}
\newtheorem{example}{Example}

\numberwithin{equation}{section}

\begin{document}
\title[The Research on Rotational Surfaces in pseudo Euclidean 4-Space with
index 2]{The Research on Rotational Surfaces in pseudo Euclidean 4-Space
with index 2}
\author{Fatma ALMAZ$^{\ast}$}
\address{Department of Mathematics, Firat university, 23119 Elazi\u{g}/T\"{U}RK\.{I}YE}
\email{fb\_fat\_almaz@hotmail.com}
\author{M\.{I}hr\.{I}ban ALYAMA\c{C} K\"{U}LAHCI}
\address{Department of Mathematics, Firat university, 23119 Elazi\u{g}/T\"{U}RK\.{I}YE}
\email{mihribankulahci@gmail.com }
\thanks{This paper is in final form and no version of it will be submitted
for publication elsewhere.}
\subjclass[2000]{53B30, 53B50, 53C80}
\keywords{Pseudo Euclidean 4-space, surfaces of rotation, killing vector
field.}

\begin{abstract}
In this study, we define a brief description of the hyperbolic and elliptic
rotational surfaces using a curve and matrices in 4-dimensional
semi-Euclidean space with index 2. That is, we provide different types of
rotational matrices, which are the subgroups of $M$ by rotating a selected
axis in $E^{4}$. Also, we choose two parameter matrices groups of rotations
and we give the matrices of rotation corresponding to the appropriate
subgroup in 4-dimensional semi-Euclidean space. Therefore, we generate
surfaces of rotation using Killing vector fields in $E_{2}^{4}$ and we give
the Gaussian curvature and the mean curvature of the surfaces of rotation.
\end{abstract}

\maketitle

\section{Introduction}

From the past to the present many studies have been done that deal with
rotational surfaces from algebraic and geometric aspects. The rotational
surfaces are parametrized with the help of the Killing vector field.
Therefore, the different types of matrices of rotations which are the
subgroups of a manifold corresponding to rotation about a chosen axis in the
arbitrary 4D-space are expressed. Hence, the two parameter matrices groups
of rotations can be chosen and the matrices of rotation corresponding to the
appropriate subgroup of an arbitrary 4D-space are expressed. To mention
briefly for the publications taken as reference related to the subject
studied. In \cite{1}, the geometric quantities associated with the concept
of surfaces and the indicatrix of a surface are discussed in
four-dimensional Galilean space by the authors. In \cite{2}, the brief
description of rotational surfaces are given using a curve and matrices in
4-dimensional (4D) Galilean space. Also, choosing two parameter matrices
groups of rotations, the matrices of rotation corresponding to the
appropriate subgroup in Galilean 4-space, the rotated surfaces are expressed
by the authors. In \cite{3,4}, the authors gave magnetic rotated surfaces in
lightlike cone $Q^{2}\subset E_{1}^{3}$. Furthermore, the conditions being
geodesic on rotational surface generated by magnetic curve are expressed
with the help of Clairaut's theorem. In \cite{5}, the representation
formulas of non-null curves are expressed in semi-Euclidean 4-space $%
E_{2}^{4}$ and some certain results of describing the nun-null normal curve
are presented in $E_{2}^{4}$. In \cite{7,8}, the rotational surfaces are
studied by different authors in Minkowski 4-space. In \cite{9}, the some
issues of displaying two-dimensional surfaces in 4D space are examined by
authors. In \cite{14}, the translation surface in the case being harmonic
surface are mainly studied, the necessary and sufficient conditions of being
semi-parallel surfaces by considering semi-parallel condition given by the
authors. In \cite{15}, the surfaces of revolution are characterized in the
three dimensional pseudo-Galilean space.

\section{Preliminaries}

Let $E_{2}^{4}$ denote the $4-$dimensional pseudo-Euclidean space with
signature $(2,4)$, that is, the real vector space $%
\mathbb{R}
^{4}$ endowed with the metric $\left\langle ,\right\rangle _{E_{2}^{4}}$
which is defined by 
\begin{equation}
\left\langle ,\right\rangle
_{E_{2}^{4}}=g=-dx_{1}^{2}-dx_{2}^{2}+dx_{3}^{2}+dx_{4}^{2},  \tag{2.1}
\end{equation}%
where $(x_{1},x_{2},x_{3},x_{4})$ is a standard rectangular coordinate
system in $E_{2}^{4}$.

Recall that an arbitrary vector $v\in E_{2}^{4}\backslash \{0\}$ can have
one of three characters: it can be space-like if $g(v,v)>0$ or $v=0,$
time-like if $g(v,v)<0$ and null if $g(v,v)=0$ and $v\neq 0.$

The norm of a vector $v$ is given by $\parallel v\parallel =\sqrt{g(v,v)}$
and two vectors $v$ and $w$ are said to be orthogonal if $g(v,w)=0$. An
arbitrary curve $x(s)$ in $E_{2}^{4}$ can locally be space-like, time-like
or null.

A space-like or time-like curve $x(s)$ has unit speed, if $g(x^{\prime
},x^{\prime })=\pm 1.$

Let $%
(x_{1},x_{2},x_{3},x_{4}),(y_{1},y_{2},y_{3},y_{4}),(z_{1},z_{2},z_{3},z_{4}) 
$ be any three vectors in $E_{2}^{4}$. The pseudo Euclidean cross product is
given as 
\begin{equation}
x\wedge y\wedge z=%
\begin{pmatrix}
-i_{1} & -i_{2} & i_{3} & i_{4} \\ 
x_{1} & x_{2} & x_{3} & x_{4} \\ 
y_{1} & y_{2} & y_{3} & y_{4} \\ 
z_{1} & z_{2} & z_{3} & z_{4}%
\end{pmatrix}%
,  \tag{2.2}
\end{equation}%
where $i_{1}=\left( 1,0,0,0\right) ,i_{2}=\left( 0,1,0,0\right)
,i_{3}=\left( 0,0,1,0\right) ,i_{4}=\left( 0,0,0,1\right) $, \cite{11,12,13}.

The pseudo-Riemannian sphere $S_{2}^{3}\left( m,r\right) $ centered at $m\in
E_{2}^{4}$ with radius $r>0$ of $E_{2}^{4}$ is defined by 
\begin{equation*}
S_{2}^{3}\left( m,r\right) =\left\{ x\in E_{2}^{4}:\left\langle
x-m,x-m\right\rangle =r^{2}\right\} .
\end{equation*}

The pseudo-hyperbolic space $H_{1}^{3}\left( m,r\right) $ centered at $m\in
E_{2}^{4}$ with radius $r>0$ of $E_{2}^{4}$ is defined by%
\begin{equation*}
H_{1}^{3}\left( m,r\right) =\left\{ x\in E_{2}^{4}:\left\langle
x-m,x-m\right\rangle =-r^{2}\right\} .
\end{equation*}

The pseudo-Riemannian sphere $S_{2}^{3}\left( m,r\right) $ is diffeomorfic
to $%
\mathbb{R}
^{2}\times S$ and the pseudo-hyperbolic space $H_{1}^{3}\left( m,r\right) $
is diffeomorfic to $S^{1}\times 
\mathbb{R}
^{2}$. The hyperbolic space $H^{3}\left( m,r\right) $ is given by%
\begin{equation*}
H^{3}\left( m,r\right) =\left\{ x\in E_{2}^{4}:\left\langle
x-m,x-m\right\rangle =-r^{2},x_{1}>0\right\} .
\end{equation*}

Let $\Psi :M\rightarrow E_{2}^{4}$ be an isometric immersion of oriented
pseudo-Riemannian submanifold $M$ into $E_{2}^{4}$. Henceforth, a
submanifold in $E_{2}^{4}$ always means pseudo-Riemannian. Let $\overset{-}{%
\nabla }$ be the Levi-Civita connection of $E_{2}^{4}$ and $\nabla $ be the
induced connection on $M$. Also, for any vector fields $X,Y$ tangent to $M$,
we get the Gaussian formula 
\begin{equation}
\overset{-}{\nabla }_{X}^{{}}Y=\nabla _{X}^{{}}Y+h(X,Y),  \tag{2.3}
\end{equation}%
where $h$ is the second fundamental form which is symmetric in $X$ and $Y$.
For a unit normal vector field $\xi $, the Weingarten formula is defined by 
\begin{equation}
\overset{-}{\nabla }_{X}^{{}}\xi =-A_{\xi }X+D_{\xi}X,  \tag{2.4}
\end{equation}%
where $A_{\xi }$ is the Weingarten map or the shape operator with respect to 
$\xi $ and $D $ is the normal connection. The Weingarten map $A_{\xi }$ is a
self-adjoint endomorphism of $TM$ which cannot be diagonalized generally. It
is known that $h$ and $A_{\xi }$ are related by 
\begin{equation}
\left\langle h(X,Y),\xi \right\rangle =\left\langle A_{\xi }X,Y\right\rangle
.  \tag{2.5}
\end{equation}

The covariant derivative $\overset{\sim }{\nabla }h$ of the second
fundamental form $h$ is given by%
\begin{equation}
\overset{\sim }{\nabla }_{X}h\left( Y,Z\right) =\nabla _{X}^{\bot }h\left(
Y,Z\right) -h\left( \nabla _{X}^{{}}Y,Z\right) -h\left( Y,\nabla
_{X}^{{}}Z\right) ,  \tag{2.6}
\end{equation}%
where $\nabla ^{\bot }$ indicates the linear connection induced on the
normal bundle $T^{\bot }M$. Also, Codazzi equation is given by%
\begin{equation}
\overset{\sim }{\nabla }_{X}h\left( Y,Z\right) =\overset{\sim }{\nabla }%
_{Y}h\left( X,Z\right) .  \tag{2.7}
\end{equation}

Let $e_{1},e_{2},...,e_{m}$ be a local orthonormal frame field in $E_{s}^{m}$
such that $e_{1},e_{2},...,e_{n}$ are tangent to $M^{n}$ and $\left\{
e_{n+1},...,e_{m}\right\} $ are normal to $M^{n}$. Let $%
w_{1},w_{2},...,w_{m} $ be the coframe of $e_{1},e_{2},...,e_{m}$. We'll
make use of the following convention on the ranges of indices $1\leq
i,j,...\leq n,n+1\leq s,t,...\leq 4,1\leq A,B,...\leq 4$. Also, $w_{A}\left(
e_{B}\right) =\delta _{AB}$ and the pseudo-Riemannian metric on $E_{s}^{m}$
is given by 
\begin{equation}
ds^{2}=\underset{i}{\overset{n}{\sum }}\varepsilon _{A}w_{A}^{2};\varepsilon
_{A}=\left\langle e_{A},e_{A}\right\rangle =\pm 1.  \tag{2.8}
\end{equation}

Let $w_{A}$ be the dual 1-form of $e_{A}$ defined by $w_{A}X=\left\langle
e_{A},X\right\rangle $. Also, the connection forms $w_{AB}$ are defined by%
\begin{equation}
de_{A}=\sum \varepsilon _{B}w_{AB}e_{B};w_{AB}+w_{BA}=0.  \tag{2.9}
\end{equation}

After, the structure equations of $E_{2}^{4}$ are written as follows%
\begin{equation}
dw_{A}=\sum_{B}\varepsilon _{B}w_{AB}\wedge w_{B};dw_{A}=\sum_{C}\varepsilon
_{C}w_{AC}\wedge w_{CB.}  \tag{2.10}
\end{equation}

The canonical forms $\left\{ w_{A}\right\} $ and the connection forms $%
\left\{ w_{AB}\right\} $ restricted to $M^{n}$ are also indicated by the
same symbols. Also, we get 
\begin{equation*}
w_{s}=0,s=n+1,...,4
\end{equation*}%
and since $w_{s}$ are zero forms on $M^{n}$, there are symmetric tensor $%
h_{ij}^{s}$ by Cartan's lemma such%
\begin{equation}
w_{is}=\sum_{j}\varepsilon _{j}h_{ij}^{s}w_{j};h_{ij}^{s}=h_{ji}^{s}. 
\tag{2.11}
\end{equation}

The mean curvature vector $H$ of $M^{n}$ in $E_{s}^{m}$ is given by 
\begin{equation}
H=\frac{1}{2}\overset{m}{\underset{s=n+1}{\sum }}\underset{i=1}{\overset{n}{%
\sum }}\varepsilon _{j}\varepsilon _{s}h_{ij}^{s}e_{s}.  \tag{2.12}
\end{equation}

Also, the covariant differentiation of $e_{i}$ is given by%
\begin{equation*}
de_{i}=\sum_{A}\varepsilon _{A}w_{iA}e_{A}\text{ or }\overset{-}{\nabla }%
_{e_{i}}^{{}}e_{j}=\sum_{B}\varepsilon _{B}w_{jB}\left( e_{i}\right) e_{B}%
\text{,}
\end{equation*}%
\cite{6,10,11}.

\begin{definition}
\cite{10}, A one-parameter group of Diffeomorphisms of a manifold $M$ is a
regular map $\psi:M\times%
\mathbb{R}
\rightarrow M$, such that $\psi_{t}(x)=\psi(x,t),$ where

\begin{enumerate}
\item $\psi_{t}:M\rightarrow M$ is a Diffeomorphism

\item $\psi_{0}=id$

\item $\psi_{s+t}=\psi_{s}o\psi_{t}.$
\end{enumerate}
\end{definition}

This group is attached with a vector field $W$ given by $\frac{d}{dt}\psi
_{t}(x)=W(x),$ and the group of Diffeomorphism is said to be as the flow of $%
W$.

\begin{definition}
If a one-parameter group of isometries is generated by a vector field $W$,
then this vector field is called as a Killing vector field, \cite{10}.
\end{definition}

\begin{definition}
Let $W$ be a vector field on a smooth manifold $M$ and $\psi_{t}$ be the
local flow generated by $W$. For each $t\in%
\mathbb{R}
,$ the map $\psi_{t}$ is Diffeomorphism of $M$ and given a function $f$ on $%
M $, we consider the Pull-back $\psi_{t}f$. We define the Lie derivative of
the function $f$ as to $W$ by%
\begin{equation*}
L_{_{W}}f=\underset{t\longrightarrow0}{\lim}\underset{}{\left( \frac{\psi
_{t}f-f}{t}\right) =\frac{d\psi_{t}f}{dt}_{t=0}}.
\end{equation*}

Let $g_{xy}$ be any pseudo-Riemannian metric, then the derivative is given
as 
\begin{equation*}
L_{_{W}}g_{xy}=g_{xy,z}W^{z}+g_{xz}W_{,y}^{z}+g_{zy}W_{,x}^{z}.
\end{equation*}

In Cartesian coordinates in Euclidean spaces where $g_{xy,z}=0,$ and the Lie
derivative is given by%
\begin{equation*}
L_{_{W}}g_{xy}=g_{xz}W_{,y}^{z}+g_{zy}W_{,x}^{z}.
\end{equation*}

In \cite{10}, the vector $W$ generates a Killing field if and if only 
\begin{equation*}
L_{_{W}}g=0.
\end{equation*}
\end{definition}

\section{The surfaces of rotation in $E_{2}^{4}$}

In this chapter, we provides a description of surfaces of rotation in $%
E_{2}^{4}$. Here, we have used the metric (2.1). Therefore, we will provide
different types of matrices of rotations, which are the subgroups of $M$ by
rotated a selected axis in $E^{4}$. Hence, we will choose two parameter
matrices groups of rotations. In particular, we have defined a brief
description of rotational surfaces in four dimensional $E_{2}^{4}$ and we
give the rotational matrices corresponding to the appropriate subgroup in $%
E_{2}^{4}$. Hence, we generate the rotational surfaces.

The rotation matrices are replaced by Lorentz transformation as follows 
\begin{equation}
M^{T}gM=g,  \tag{3.1}
\end{equation}%
where $M^{T}$ is the transpoze, $g$ is the metric matrix of $E_{2}^{4}$ and
for the metric (2.1).

Let's obtain the set of all $4\times 4$ type matrices satisfying (3.1). The
Lorentz group is a subgroup of the Diffeomorphisms group in $E_{2}^{4}.$

\begin{theorem}
Let the pseudo-Euclidean group be a subgroup of the Diffeomorphisms group in 
$E_{2}^{4}$ and let $W$ be vector field which generate the isometries. Then,
the killing vector field associated with the metric $g$ is given as%
\begin{equation*}
W(\xi ,\varrho ,\vartheta ,\eta )=a\left( \eta \partial \xi +\xi \partial
\eta \right) +b\left( \vartheta \partial \varrho +\varrho \partial \vartheta
\right) +c\left( \vartheta \partial \xi +\xi \partial \vartheta \right)
\end{equation*}%
\begin{equation}
+d(\eta \partial \varrho +\varrho \partial \eta )+e(\vartheta \partial \eta
-\eta \partial \vartheta )+f\left( \xi \partial \varrho -\varrho \partial
\xi \right) ,  \tag{3.2}
\end{equation}%
where $a,b,c,d,e,f\in 
\mathbb{R}
_{0}^{+}.$
\end{theorem}

\begin{proof}
Let $W$ be the vector which generate the isometries in $E_{2}^{4}$. We can
write as the following the general vector field; 
\begin{equation}
W(\xi ,\varrho ,\vartheta ,\eta )=W^{1}(\xi ,\varrho ,\vartheta ,\eta
)\partial \xi +W^{2}(\xi ,\varrho ,\vartheta ,\eta )\partial \varrho
+W^{3}(\xi ,\varrho ,\vartheta ,\eta )\partial \vartheta +W^{4}(\xi ,\varrho
,\vartheta ,\eta )\partial \eta ,  \tag{3.3}
\end{equation}%
where $W^{j}\ $are real functions $($for $j=1,2,3,4)$. Also, by using
definition 2 and definition 3, the expression of the (3.3) is%
\begin{equation}
W_{\xi }^{1}=W_{\varrho }^{2}=W_{\vartheta }^{3}=W_{\eta }^{4}=0,  \tag{3.4}
\end{equation}%
\begin{equation}
W_{\varrho }^{1}+W_{\xi }^{2}=0;W_{\vartheta }^{1}-W_{\xi }^{3}=0;W_{\eta
}^{1}-W_{\xi }^{4}=0  \tag{3.5}
\end{equation}%
\begin{equation}
W_{\vartheta }^{2}-W_{\varrho }^{3}=0,W_{\eta }^{2}-W_{\varrho
}^{4}=0,W_{\eta }^{3}+W_{\vartheta }^{4}=0,  \tag{3.6}
\end{equation}%
first, we will obtain the function $W^{1}$, then from (3.4) and (3.5) we
write 
\begin{equation}
W_{\varrho }^{1}+W_{\xi }^{2}=0.  \tag{3.7}
\end{equation}%
then differentiating with respect to $\varrho $ in\ the previous equation
(3.7), we have%
\begin{equation}
W_{\varrho \varrho }^{1}+W_{\xi \varrho }^{2}=0  \tag{3.8}
\end{equation}%
and then differentiating with respect to $\vartheta $ in the equations $%
W_{\vartheta }^{1}-W_{\xi }^{3}=0$, we obtain%
\begin{equation}
W_{\vartheta \vartheta }^{1}-W_{\xi \vartheta }^{3}=0,  \tag{3.9a}
\end{equation}%
and then differentiating with respect to $\eta $ in the equations $W_{\eta
}^{1}-W_{\xi }^{4}=0$, we obtain%
\begin{equation}
W_{\eta \eta }^{1}-W_{\xi \eta }^{4}=0  \tag{3.9b}
\end{equation}%
and from (3.4) we get $W_{\xi \varrho }^{2},W_{\xi \eta }^{4},W_{\xi
\vartheta }^{3}=0.$ From (3.8) and (3.9a), (3.9b) which gives $W_{\vartheta
\vartheta }^{1},W_{\eta \eta }^{1},W_{\varrho \varrho }^{1}=0.$ Therefore,
the function $W^{1}$ can be written as follows%
\begin{align}
W^{1}(\varrho ,\vartheta ,\eta )& =f_{1}^{1}(\varrho ,\eta )\vartheta
+g_{1}^{1}(\varrho ,\eta ),  \tag{3.10} \\
W^{1}(\varrho ,\vartheta ,\eta )& =f_{1}^{2}(\varrho ,\vartheta )\eta
+g_{1}^{2}(\varrho ,\vartheta ),  \tag{3.11} \\
W^{1}(\varrho ,\vartheta ,\eta )& =f_{1}^{3}(\eta ,\vartheta )\varrho
+g_{1}^{3}(\eta ,\vartheta ),  \notag
\end{align}%
where $f_{1}^{i},g_{1}^{i}\in C^{\infty },i=\{i=1,2,3\}.$ From (3.10) and
since $W_{\eta \eta }^{1}=0$ we get 
\begin{equation*}
W_{\eta \eta }^{1}(\varrho ,\vartheta ,\eta )=f_{1\eta \eta }^{1}(\varrho
,\eta )\vartheta +g_{1\eta \eta }^{1}(\varrho ,\eta )=0,
\end{equation*}%
this means $f_{1\eta \eta }^{1}(\varrho ,\eta )=g_{1\eta \eta }^{1}(\varrho
,\eta )=$ $0.$ Thus, we can write the equations $f_{1}^{1}(\varrho ,\eta )$
and $g_{1}^{1}(\varrho ,\eta )$ as follows 
\begin{align*}
f_{1}^{1}(\varrho ,\eta )& =h_{1}(\varrho )\eta +m_{1}(\varrho ), \\
g_{1}^{1}(\varrho ,\eta )& =h_{1}^{\ast }(\varrho )\eta +m_{2}(\varrho ).
\end{align*}

Furthermore, since $W_{\varrho \varrho }^{1}=0$ we can choose the functions $%
h_{1},h_{1}^{\ast },m_{1},m_{2}$ as 
\begin{align}
h_{1}(\varrho )& =a_{1}\varrho +b_{1},h_{1}^{\ast }(\varrho )=a_{1}^{\ast
}\varrho +b_{2},  \notag \\
m_{1}(\varrho )& =c_{1}\varrho +d_{1},m_{2}(\varrho )=c_{2}\varrho +d_{2}. 
\tag{3.12}
\end{align}

Furthermore, substituting this equation into (3.10), we have%
\begin{equation}
W^{1}(\varrho ,\vartheta ,\eta )=\left( \left( a_{1}\varrho +b_{1}\right)
\eta +c_{1}\varrho +d_{1}\right) \vartheta +\left( a_{1}^{\ast }\varrho
+b_{2}\right) \eta +c_{2}\varrho +d_{2}.  \tag{3.13}
\end{equation}

Similarly, by making the necessary algebraic operations, the following
component equations are obtained, respectively.%
\begin{eqnarray*}
W^{2}(\xi ,\vartheta ,\eta ) &=&\left( \left( a_{2}\vartheta +b\right) \eta
+x_{2}^{1}\vartheta +y_{2}^{1}\right) \xi +\left( a_{2}^{\ast }\vartheta
+b^{\ast }\right) \eta +x_{2}^{2}\vartheta +y_{2}^{2}; \\
W^{3}(\varrho ,\xi ,\eta ) &=&\left( \left( a_{3}\eta +d\right) \xi
+x_{3}^{1}\eta +y_{3}^{1}\right) \varrho +\left( a_{3}^{\ast }\eta +d^{\ast
}\right) \xi +x_{3}^{2}\eta +y_{3}^{2}; \\
W^{4}(\xi ,\vartheta ,\varrho ) &=&\left( \left( a_{4}\varrho +e\right)
\vartheta +x_{4}^{1}\varrho +y_{4}^{1}\right) \xi +\left( a_{4}^{\ast
}\varrho +e^{\ast }\right) \vartheta +x_{4}^{2}\varrho +y_{4}^{2},
\end{eqnarray*}
where $a_{i},a_{i}^{\ast },x_{i}^{j},y_{i}^{j},b,d,e,b^{\ast },d^{\ast
},e^{\ast }\in 
\mathbb{R}
;i,j\in I.$

If we assume arbitrary constants as 
\begin{equation*}
a_{i}=x_{i}^{1}=c_{1}=b_{1}=b=d=e=a_{i}^{\ast }=y_{i}^{2}=d_{2}=0;i\in
\{1,2,3,4\}
\end{equation*}%
then we obtain%
\begin{align*}
W^{1}(\varrho ,\vartheta ,\eta )& =d_{1}\vartheta +b_{2}\eta +c_{2}\varrho
.;W^{2}(\xi ,\vartheta ,\eta )=y_{2}^{1}\xi +b^{\ast }\eta
+x_{2}^{2}\vartheta; \\
W^{3}(\varrho ,\xi ,\eta )& =y_{3}^{1}\varrho +d^{\ast }\xi +x_{3}^{2}\eta
;W^{4}(\xi ,\vartheta ,\varrho )=y_{4}^{1}\xi +e^{\ast }\vartheta
+x_{4}^{2}\varrho .
\end{align*}

Furthermore, by using the equations (3.4), (3.5) and (3.6), we write%
\begin{eqnarray*}
y_{2}^{1} &=&-c_{2}=f;d_{1}=d^{\ast }=c;b_{2}=y_{4}^{1}=a; \\
x_{2}^{2} &=&y_{3}^{1}=b;b^{\ast }=x_{4}^{2}=d;e^{\ast }=-x_{3}^{2}=e.
\end{eqnarray*}

Hence, the vector fields $W^{1},W^{2},W^{3},W^{4}$ are given by%
\begin{align}
W^{1}(\varrho ,\vartheta ,\eta )& =c\vartheta +a\eta -f\varrho .;W^{2}(\xi
,\vartheta ,\eta )=f\xi +d\eta +b\vartheta ;  \tag{3.14} \\
W^{3}(\varrho ,\xi ,\eta )& =b\varrho +c\xi -e\eta ;W^{4}(\xi ,\vartheta
,\varrho )=a\xi +e\vartheta +d\varrho .  \notag
\end{align}

By using the equation (3.14) into the equation (3.3), we have%
\begin{equation*}
W(\xi ,\varrho ,\vartheta ,\eta )=\left( c\vartheta +a\eta -f\varrho \right)
\partial \xi +\left( f\xi +d\eta +b\vartheta \right) \partial \varrho
+\left( b\varrho +c\xi -e\eta \right) \partial \vartheta
\end{equation*}%
\begin{equation*}
+\left( a\xi +e\vartheta +d\varrho \right) \partial \eta ;
\end{equation*}%
\begin{equation*}
W(\xi ,\varrho ,\vartheta ,\eta )=a\left( \eta \partial \xi +\xi \partial
\eta \right) +b\left( \vartheta \partial \varrho +\varrho \partial \vartheta
\right) +c\left( \vartheta \partial \xi +\xi \partial \vartheta \right)
\end{equation*}%
\begin{equation*}
+d(\eta \partial \varrho +\varrho \partial \eta )+e(\vartheta \partial \eta
-\eta \partial \vartheta )+f\left( \xi \partial \varrho -\varrho \partial
\xi \right) ,
\end{equation*}%
where $a,b,c,d,e,f\in 
\mathbb{R}
_{0}^{+}.$
\end{proof}

\begin{theorem}
Let $W(\xi ,\varrho ,\vartheta ,\eta )$ be the killing vector field and let $%
\gamma=(f_{1},f_{2},f_{3},f_{4})$ be a curve in $E_{2}^{4}$, then the
surfaces of rotation are given as follows

\begin{enumerate}
\item For the rotations $\Omega _{1}=\vartheta \partial \xi +\xi \partial
\vartheta $ and $\Omega _{4}=\eta \partial \varrho +\varrho \partial \eta ,$
the hyperbolic surface of rotation is given as%
\begin{equation*}
S_{14}(x,\alpha ,s)=\left( 
\begin{array}{c}
f_{1}\cosh x+f_{3}\sinh x,f_{2}\cosh \alpha +f_{4}\sinh \alpha , \\ 
f_{1}\sinh x+f_{3}\cosh x,f_{2}\sinh \alpha +f_{4}\cosh \alpha 
\end{array}%
\right) 
\end{equation*}%
and for the curve $\gamma (s)=(f_{1}(s),0,0,f_{4}(s))$ the Gaussian
curvature $K$ and the mean curvature vector $H$ of the rotational surface $%
S_{14}(x(t),\alpha (t),s)=\left( f_{1}\cosh x,f_{4}\sinh \alpha ,f_{1}\sinh
x,f_{4}\cosh \alpha \right) $ are given as%
\begin{equation*}
K=\frac{\left( f_{1}^{\prime }f_{4}^{{}}-f_{1}f_{4}^{\prime }\right)
^{2}\left( \overset{.}{x}\overset{.}{\alpha }\right) ^{2}}{f_{4}^{2}\overset{%
.}{\alpha }^{2}-f_{1}^{2}\overset{.}{x}^{2}}+\frac{\left( f_{1}^{\prime
}f_{4}\overset{.}{\alpha }^{2}-f_{4}^{\prime }f_{1}\overset{.}{x}^{2}\right)
\left( f_{1}^{\prime }f_{4}^{\prime \prime }-f_{1}^{\prime \prime
}f_{4}^{\prime }\right) }{-f_{1}^{\prime 2}+f_{4}^{\prime 2}},
\end{equation*}%
\begin{equation*}
H=\{\frac{f_{1}f_{4}\left( \overset{..}{x}\overset{.}{\alpha }+\overset{.}{x}%
\overset{..}{\alpha }\right) }{2\sqrt{f_{4}^{2}\overset{.}{\alpha }%
^{2}-f_{1}^{2}\overset{.}{x}^{2}}}+\frac{f_{4}^{\prime }f_{1}\overset{.}{x}%
^{2}-f_{1}^{\prime }f_{4}\overset{.}{\alpha }^{2}}{2\sqrt{-f_{1}^{\prime
2}+f_{4}^{\prime 2}}}\}e_{3}+\frac{\left( f_{1}^{\prime }f_{4}^{\prime
\prime }-f_{1}^{\prime \prime }f_{4}^{\prime }\right) }{2\sqrt{%
-f_{1}^{\prime 2}+f_{4}^{\prime 2}}}e_{4}
\end{equation*}%
where 

$e_{3}=\frac{\left( f_{4}\overset{.}{\alpha }\sinh x,f_{1}\overset{.}{x%
}\cosh \alpha ,f_{4}\overset{.}{\alpha }\cosh x,f_{1}\overset{.}{x}\sinh
\alpha \right) }{\sqrt{f_{4}^{2}\overset{.}{\alpha }^{2}-f_{1}^{2}\overset{.}%
{x}^{2}}},e_{4}=\frac{\left( f_{4}^{\prime }\cosh x,f_{1}^{\prime }\sinh
\alpha ,f_{4}^{\prime }\sinh x,f_{1}^{\prime }\cosh \alpha \right) }{\sqrt{%
-f_{1}^{\prime 2}+f_{4}^{\prime 2}}}.$

\item For the rotations $\Omega _{2}=\eta \partial \xi +\xi \partial \eta $
and $\Omega _{3}=\vartheta \partial \varrho +\varrho \partial \vartheta ,$
the hyperbolic surface of rotation is given as 
\begin{equation*}
S_{23}(y,z,s)=\left( 
\begin{array}{c}
f_{1}\cosh y+f_{4}\sinh y,f_{2}\cosh z+f_{3}\sinh z, \\ 
f_{2}\sinh z+f_{3}\cosh z,f_{1}\sinh y+f_{4}\cosh y%
\end{array}%
\right) .
\end{equation*}%
and for the curve $\gamma (s)=(f_{1}(s),f_{2}(s),0,0)$ the Gaussian
curvature $K$ and the mean curvature vector $H$ of the rotational surface $%
S_{23}(y(t),z(t),s)=\left( f_{1}\cosh y,f_{2}\cosh z,f_{2}\sinh z,f_{1}\sinh
y\right) $ are given as%
\begin{equation*}
K=-\left( 
\begin{array}{c}
\frac{\left( f_{1}f_{2}^{\prime }+f_{1}^{\prime }f_{2}^{{}}\right)
^{2}\left( \overset{.}{y}\overset{.}{z}\right) ^{2}}{f_{2}^{2}\overset{.}{z}%
+f_{1}^{2}\overset{.}{y}}+ \\ 
\frac{\left( f_{1}^{{}}f_{2}^{\prime }\overset{.}{y}^{2}+f_{1}^{\prime }f_{2}%
\overset{.}{z}^{2}\right) \left( f_{1}^{\prime \prime }f_{2}^{\prime
}+f_{1}^{\prime }f_{2}^{\prime \prime }\right) }{f_{1}^{\prime
2}+f_{2}^{\prime 2}}%
\end{array}%
\right) ;H=\left( 
\begin{array}{c}
\frac{f_{1}^{{}}f_{2}(\overset{.}{y}\overset{..}{z}+\overset{..}{y}\overset{.%
}{z})}{2\sqrt{f_{2}^{2}\overset{.}{z}+f_{1}^{2}\overset{.}{y}}}e_{3} \\ 
+\frac{f_{1}^{{}}f_{2}^{\prime }\overset{.}{y}^{2}+f_{1}^{\prime }f_{2}%
\overset{.}{z}^{2}-f_{1}^{\prime \prime }f_{2}^{\prime }-f_{1}^{\prime
}f_{2}^{\prime \prime }}{2\sqrt{f_{1}^{\prime 2}+f_{2}^{\prime 2}}}e_{4}%
\end{array}%
\right) ,
\end{equation*}%
where 

$e_{3}=\frac{\left( f_{2}\overset{.}{z}\sinh y,f_{1}\overset{.}{y}%
\sinh z,f_{1}\overset{.}{y}\cosh z,f_{2}\overset{.}{z}\cosh y\right) }{\sqrt{%
f_{2}^{2}\overset{.}{z}+f_{1}^{2}\overset{.}{y}}}$; $e_{4}=\frac{\left(
f_{2}^{\prime }\cosh y,f_{1}^{\prime }\cosh z,f_{1}^{\prime }\sinh
z,f_{2}^{\prime }\sinh y\right) }{\sqrt{f_{1}^{\prime 2}+f_{2}^{\prime 2}}}$

\item For the rotations $\Omega _{5}=\xi \partial \varrho -\varrho \partial
\xi $ and $\Omega _{6}=\vartheta \partial \eta -\eta \partial \vartheta ,$
the elliptic surface of rotation is given as 
\begin{equation*}
S_{56}(\beta ,\theta ,s)=\left( 
\begin{array}{c}
f_{1}\cos \beta +f_{2}\sin \beta ,-f_{1}\sin \beta +f_{2}\cos \beta , \\ 
f_{3}\cos \theta +f_{4}\sin \theta ,-f_{3}\sin \theta +f_{4}\cos \theta 
\end{array}%
\right) ,
\end{equation*}%
and for the curve $\gamma (s)=(0,f_{2}(s),0,f_{4}(s))$ the Gaussian
curvature $K$ and the mean curvature vector $H$ of the rotational surface $%
S_{56}(\beta \left( t\right) ,\theta \left( t\right) ,s)=\left( f_{2}\sin
\beta ,f_{2}\cos \beta ,f_{4}\sin \theta ,f_{4}\cos \theta \right) $ are
given as 
\begin{equation*}
K=-\left( 
\begin{array}{c}
\frac{\left( f_{2}^{\prime }f_{4}-f_{2}f_{4}^{\prime }\right) ^{2}\left( 
\overset{.}{\beta }\overset{.}{\theta }\right) ^{2}}{-f_{2}^{2}\overset{.}{%
\beta }^{2}+f_{4}^{2}\overset{.}{\theta }}+ \\ 
\frac{\left( -f_{2}^{\prime \prime }f_{4}^{\prime }+f_{2}^{\prime
}f_{4}^{\prime \prime }\right) (f_{4}^{\prime }f_{2}\overset{.}{\beta }%
^{2}-f_{2}^{\prime }f_{4}\overset{.}{\theta }^{2})^{2}}{-f_{2}^{\prime
2}+f_{4}^{\prime 2}}%
\end{array}%
\right) ;
\end{equation*}
\begin{equation*}
H=\frac{f_{4}f_{2}\left( \overset{.}{\beta }\overset{..}{\theta }-\overset{.}%
{\theta }\overset{..}{\beta }\right) }{2\sqrt{f_{4}^{2}\overset{.}{\theta }%
-f_{2}^{2}\overset{.}{\beta }^{2}}}e_{3}+\frac{\left( f_{4}^{\prime }f_{2}%
\overset{.}{\beta }^{2}-f_{2}^{\prime }f_{4}\overset{.}{\theta }%
^{2}+f_{2}^{\prime \prime }f_{4}^{\prime }-f_{2}^{\prime }f_{4}^{\prime
\prime }\right) }{2\sqrt{f_{4}^{\prime 2}-f_{2}^{\prime 2}}}e_{4}
\end{equation*}
where
$e_{3}=\frac{\left( -f_{4}\overset{.}{\theta }\cos \beta ,f_{4}\overset%
{.}{\theta }\sin \beta ,-f_{2}\overset{.}{\beta }\cos \theta ,f_{2}\overset{.%
}{\beta }\sin \theta \right) }{\sqrt{-f_{2}^{2}\overset{.}{\beta }%
^{2}+f_{4}^{2}\overset{.}{\theta }}},e_{4}=\frac{\left( f_{4}^{\prime }\sin
\beta ,f_{4}^{\prime }\cos \beta ,f_{2}^{\prime }\sin \theta ,f_{2}^{\prime
}\cos \theta \right) }{\sqrt{-f_{2}^{\prime 2}+f_{4}^{\prime 2}}}$; $ -\infty
<x,y,z,\alpha ,\beta ,\theta <\infty ,s\in I$ and $f_{i}\in C^{\infty }.$
\end{enumerate}
\end{theorem}

\begin{proof}
Let $W(\xi ,\varrho ,\vartheta ,\eta )=a\Omega _{2}+b\Omega _{3}+c\Omega
_{1}+d\Omega _{4}+e$ $\Omega _{6}+f\Omega _{5}$ be the killing vector field.
Hence, we can give vector fields generating the rotations as follows 
\begin{equation}
\Omega _{1}=\vartheta \partial \xi +\xi \partial \vartheta ;\text{ }\Omega
_{2}=\eta \partial \xi +\xi \partial \eta ;\text{ }\Omega _{3}=\vartheta
\partial \varrho +\varrho \partial \vartheta ;  \tag{3.15a}
\end{equation}%
\begin{equation}
\Omega _{4}=\eta \partial \varrho +\varrho \partial \eta ;\text{ }\Omega
_{5}=\xi \partial \varrho -\varrho \partial \xi ;\text{ }\Omega
_{6}=\vartheta \partial \eta -\eta \partial \vartheta ,  \tag{3.15b}
\end{equation}%
by using the equations (3.15), we will find $4\times 4$ matrices of
hyperbolic and elliptic by rotating $\Omega _{i}$, $i\in I$.

a) Hyperbolic matrices: we give some one-parameter hyperbolic matrices
groups of rotation $\Omega _{i},i=1,2,3,4.$

$1)$ For $\Omega _{1}=\vartheta \partial \xi +\xi \partial \vartheta ,$ we
write the vector field%
\begin{equation}
\Lambda _{\Omega _{1}}=%
\begin{bmatrix}
\vartheta & 0 & \xi & 0%
\end{bmatrix}%
^{\intercal },  \tag{3.16}
\end{equation}%
then, the previous equation can be given as follows%
\begin{equation}
\Delta _{\Lambda _{\Omega _{1}}}=%
\begin{bmatrix}
0 & 0 & 1 & 0 \\ 
0 & 0 & 0 & 0 \\ 
1 & 0 & 0 & 0 \\ 
0 & 0 & 0 & 0%
\end{bmatrix}%
,  \tag{3.17}
\end{equation}%
from definition 1, by using the differential equation $\frac{d}{dw}\psi
_{w}(x)=W(x)$ we have 
\begin{equation*}
\Pi _{w}(x)=e^{\Delta _{\Lambda _{1}}x}(x)=I_{4\times 4}+\Delta _{\Lambda
_{\Omega _{1}}}x+\frac{\left( \Delta _{\Lambda _{\Omega _{1}}}x\right) ^{2}}{%
2!}+...
\end{equation*}%
\begin{equation}
\Pi _{\Omega _{1}}(x)=%
\begin{bmatrix}
\cosh x & 0 & \sinh x & 0 \\ 
0 & 1 & 0 & 0 \\ 
\sinh x & 0 & \cosh x & 0 \\ 
0 & 0 & 0 & 1%
\end{bmatrix}%
.  \tag{3.18}
\end{equation}%
$2)$ For $\Omega _{2}=\eta \partial \xi +\xi \partial \eta ,$ we write the
vector field%
\begin{equation}
\Lambda _{\Omega _{2}}=%
\begin{bmatrix}
\eta & 0 & 0 & \xi%
\end{bmatrix}%
^{\intercal },  \tag{3.19}
\end{equation}%
then, the previous equation can be given as follows 
\begin{equation}
\Delta _{\Lambda _{\Omega _{2}}}=%
\begin{bmatrix}
0 & 0 & 0 & 1 \\ 
0 & 0 & 0 & 0 \\ 
0 & 0 & 0 & 0 \\ 
1 & 0 & 0 & 0%
\end{bmatrix}%
,  \tag{3.20}
\end{equation}%
from definition 1, by using the differential equation $\frac{d}{du}\psi
_{u}(y)=W(y),$ we have 
\begin{equation*}
\Pi _{u}(y)=e^{\Delta _{\Lambda _{\Omega _{2}}}u}(y)=I_{4\times 4}+\Delta
_{\Lambda _{\Omega _{2}}}y+\frac{\left( \Delta _{\Lambda _{\Omega
_{2}}}y\right) ^{2}}{2!}+...
\end{equation*}%
\begin{equation}
\Pi _{\Omega _{2}}(y)=%
\begin{bmatrix}
\cosh y & 0 & 0 & \sinh y \\ 
0 & 1 & 0 & 0 \\ 
0 & 0 & 1 & 0 \\ 
\sinh y & 0 & 0 & \cosh y%
\end{bmatrix}%
.  \tag{3.21}
\end{equation}

3) For $\Omega _{3}=\vartheta \partial \varrho +\varrho \partial \vartheta ,$
we write the vector field given as 
\begin{equation}
\Lambda _{\Omega _{3}}=%
\begin{bmatrix}
0 & \vartheta & \varrho & 0%
\end{bmatrix}%
^{\intercal },  \tag{3.22}
\end{equation}%
then, the previous equation can be given as follows 
\begin{equation*}
\Delta _{\Lambda _{\Omega _{3}}}=%
\begin{bmatrix}
0 & 0 & 0 & 0 \\ 
0 & 0 & 1 & 0 \\ 
0 & 1 & 0 & 0 \\ 
0 & 0 & 0 & 0%
\end{bmatrix}%
.
\end{equation*}

Now, from definition 1 we can say that the one-parameter group of
homomorphism $\psi _{z}(\xi ,\varrho ,\vartheta ,\varsigma )$ is expressed
by $\psi _{z}^{\prime }(\xi )=\psi ^{\xi }\psi _{z}(\xi )$. So, we find $%
\psi _{z}(\xi )=e^{v\psi _{z}}\xi $ and calculating the matrix exponential,
we have 
\begin{equation*}
\Delta _{v}(z)=e^{\Delta _{\Lambda _{\Omega _{3}}}z}(z)=I_{4\times 4}+\Delta
_{\Lambda _{\Omega _{3}}}z+\frac{\left( \Delta _{\Lambda _{\Omega
_{3}}}z\right) ^{2}}{2!}+...
\end{equation*}%
\begin{equation}
\Pi _{\Omega _{3}}(z)=%
\begin{bmatrix}
1 & 0 & 0 & 0 \\ 
0 & \cosh z & \sinh z & 0 \\ 
0 & \sinh z & \cosh z & 0 \\ 
0 & 0 & 0 & 1%
\end{bmatrix}%
.  \tag{3.23}
\end{equation}

Similarly for $\Omega _{4}=\eta \partial \varrho +\varrho \partial \eta $,
we get 
\begin{equation*}
\Pi _{\Omega _{4}}(\alpha )=%
\begin{bmatrix}
1 & 0 & 0 & 0 \\ 
0 & \cosh \alpha & 0 & \sinh \alpha \\ 
0 & 0 & 1 & 0 \\ 
0 & \sinh \alpha & 0 & \cosh \alpha%
\end{bmatrix}%
,
\end{equation*}%
and for $\Omega _{5}=\xi \partial \varrho -\varrho \partial \xi $ and $%
\Omega _{6}=\vartheta \partial \eta -\eta \partial \vartheta $, we obtain
two one-parameter matrix group of rotation.

b) Elliptic matrices: we give some one-parameter elliptic matrices groups of
rotation $\Omega _{5}$ and $\Omega _{6}$%
\begin{equation*}
\Pi _{\Omega _{5}}(\beta )=%
\begin{bmatrix}
\cos \beta & \sin \beta & 0 & 0 \\ 
-\sin \beta & \cos \beta & 0 & 0 \\ 
0 & 0 & 1 & 0 \\ 
0 & 0 & 0 & 1%
\end{bmatrix}%
;\Pi _{\Omega _{6}}(\theta )=%
\begin{bmatrix}
1 & 0 & 0 & 0 \\ 
0 & 1 & 0 & 0 \\ 
0 & 0 & \cos \theta & \sin \theta \\ 
0 & 0 & -\sin \theta & \cos \theta%
\end{bmatrix}%
.
\end{equation*}

Now if we want to express surfaces of rotation generated by two hyperbolic
and elliptic subgroups, the sub-algebra of the lie algebra of the Lorentz
group can be obtained, then we can write the closed subgroups of Lorentz
group. Hence, two parameter subgroups of $SO(4,2)$ are obtain, and two
parameter subgroups that fix some axis of rotation can be expressed.
Therefore, we can write 2D sub-algebras, and therefore we need to obtain two
vectors. In this context, by using Poisson bracket of two vectors $X=\overset%
{n}{\underset{i=1}{\sum }}X^{i}\partial _{i},Y=\overset{n}{\underset{i=1}{%
\sum }}Y^{i}\partial _{i}$ defined by 
\begin{equation*}
\left[ X,Y\right] =\overset{n}{\underset{i=1}{\sum }}\overset{n}{\underset{%
j=1}{\sum }}(X^{j}\partial _{j}Y^{i}-Y^{j}\partial _{j}X^{i})\partial _{i},
\end{equation*}%
we can write the following expressions%
\begin{eqnarray*}
\left[ \Omega _{1},\Omega _{2}\right]  &=&\Omega _{6};\left[ \Omega
_{1},\Omega _{3}\right] =\Omega _{5};\left[ \Omega _{1},\Omega _{5}\right]
=\Omega _{3};\left[ \Omega _{1},\Omega _{6}\right] =\Omega _{2}; \\
\left[ \Omega _{2},\Omega _{4}\right]  &=&\Omega _{5};\left[ \Omega
_{2},\Omega _{5}\right] =\Omega _{4};\left[ \Omega _{6},\Omega _{2}\right]
=\Omega _{1};\left[ \Omega _{3},\Omega _{4}\right] =\Omega _{6}; \\
\left[ \Omega _{5},\Omega _{3}\right]  &=&\Omega _{1};\left[ \Omega
_{3},\Omega _{6}\right] =\Omega _{4};\left[ \Omega _{5},\Omega _{4}\right]
=\Omega _{2};\left[ \Omega _{6},\Omega _{4}\right] =\Omega _{3}
\end{eqnarray*}%
then these Poisson brackets are not in $Sp\{\Omega _{i},\Omega _{j}\}$
excluding $Sp\{\Omega _{1},\Omega _{4}\},Sp\{\Omega _{2},\Omega _{3}\}$ and $%
Sp\{\Omega _{5},\Omega _{6}\}.$ Therefore, these are not closed sub-algebra.
Also, 
\begin{equation*}
\left[ \Omega _{1},\Omega _{4}\right] =\left[ \Omega _{2},\Omega _{3}\right]
=\left[ \Omega _{5},\Omega _{6}\right] =0,
\end{equation*}%
$\{\Omega _{1},\Omega _{4}\},\left\{ \Omega _{2},\Omega _{3}\right\}
,\left\{ \Omega _{5},\Omega _{6}\right\} $ are the closed sub-algebra and we
can think $\{\Omega _{1},\Omega _{4}\},$ $\left\{ \Omega _{2},\Omega
_{3}\right\} ,$ $\left\{ \Omega _{5},\Omega _{6}\right\} $ as basis. Thus,
abelian subgroups of $SO(2,2)$ can be expressed. Then, $\Omega _{1},\Omega
_{4}$ and $\Omega _{2},\Omega _{3}$ generate abelian sub-algebras being
hyperbolic. Therefore, we can write matrices $\Pi _{\Omega _{1}}(x)\Pi
_{\Omega _{4}}(\alpha )$ and $\Pi _{\Omega _{2}}(y)\Pi _{\Omega _{3}}(z)$
being the rotational groups of matrices. Hence, these subgroups don't fix
any axis and so it is not a rotation about any axis. First, for the
rotations $\Omega _{1}$ and $\Omega _{4},$ the matrices of rotations of this
surface can be written as $\Pi _{\Omega _{1}}(x)\Pi _{\Omega _{4}}(\alpha )$%
. We are interested in taking a planar curve $\gamma $ with $s$ parameter as
follows 
\begin{equation}
\gamma (s)=(f_{1}(s),f_{2}(s),f_{3}(s),f_{4}(s)),s\in I  \tag{3.24}
\end{equation}%
and rotating it with 2D subgroup of isometry. Hence, the surface of
revolution $S_{14}$ around $\Pi _{\Omega _{1}}(x)$ and $\Pi _{\Omega
_{4}}(\alpha )$ can be parametrized as follows%
\begin{equation}
S_{14}(x,\alpha ,s)=\Pi _{\Omega _{1}}(x).\Pi _{\Omega _{4}}(\alpha ).%
\begin{bmatrix}
f_{1}(s) \\ 
f_{2}(s) \\ 
f_{3}(s) \\ 
f_{4}(s)%
\end{bmatrix}%
=\left( 
\begin{array}{c}
f_{1}\cosh x+f_{3}\sinh x, \\ 
f_{2}\cosh \alpha +f_{4}\sinh \alpha , \\ 
f_{1}\sinh x+f_{3}\cosh x, \\ 
f_{2}\sinh \alpha +f_{4}\cosh \alpha 
\end{array}%
\right) ,  \tag{3.25a}
\end{equation}%
where for $i\in \{1,2,3,4\},$ $f_{i}$ are smooth functions and $-\infty
<x,\alpha <\infty ,s\in I.$ Now, we consider the following rotational surface%
\begin{equation}
S^{_{^{14}}}(x(t),\alpha (t),s)=\left( f_{1}\cosh x(t),f_{4}\sinh \alpha
(t),f_{1}\sinh x(t),f_{4}\cosh \alpha (t)\right) ,  \tag{3.25b}
\end{equation}%
where $f_{1}$ and $f_{4}$ are nonzero smooth functions and the curve $\gamma
(s)=(f_{1}(s),0,f_{4}(s))$ lies on the $\xi \eta -$plane. For the rotational
surface (3.25b) we have the parametrizations 
\begin{eqnarray*}
S_{s}^{_{^{14}}}(x,\alpha ,s) &=&\left( f_{1}^{\prime }\cosh x,f_{4}^{\prime
}\sinh \alpha ,f_{1}^{\prime }\sinh x,f_{4}^{\prime }\cosh \alpha \right) ;
\\
S_{ss}^{_{^{14}}}(x,\alpha ,s) &=&\left( f_{1}^{\prime \prime }\cosh
x,f_{4}^{\prime \prime }\sinh \alpha ,f_{1}^{\prime \prime }\sinh
x,f_{4}^{\prime \prime }\cosh \alpha \right) ; \\
S_{t}^{_{^{14}}}(x,\alpha ,s) &=&\left( f_{1}\overset{.}{x}\sinh x,f_{4}%
\overset{.}{\alpha }\cosh \alpha ,f_{1}\overset{.}{x}\cosh x,f_{4}\overset{.}%
{\alpha }\sinh \alpha \right) ; \\
S_{tt}^{_{^{14}}}(x,\alpha ,s) &=&\left( 
\begin{array}{c}
f_{1}(\overset{..}{x}\sinh x+\overset{.}{x}^{2}\cosh x),f_{4}(\overset{..}{%
\alpha }\cosh \alpha +\overset{.}{\alpha }^{2}\sinh \alpha ), \\ 
f_{1}(\overset{..}{x}\cosh x+\overset{.}{x}^{2}\sinh x),f_{4}(\overset{..}{%
\alpha }\sinh \alpha +\overset{.}{\alpha }^{2}\cosh \alpha )%
\end{array}%
\right)  \\
S_{ts}^{_{^{14}}}(x,\alpha ,s) &=&\left( f_{1}^{\prime }\overset{.}{x}\sinh
x(t),f_{4}^{\prime }\overset{.}{\alpha }\cosh \alpha (t),f_{1}^{\prime }%
\overset{.}{x}\cosh x(t),f_{4}^{\prime }\overset{.}{\alpha }\sinh \alpha
\right) 
\end{eqnarray*}%
and 
\begin{equation*}
\left\langle S_{s}^{_{^{14}}},S_{s}^{_{^{14}}}\right\rangle =-f_{1}^{\prime
2}+f_{4}^{\prime 2}>0;\left\langle
S_{t}^{_{^{14}}},S_{t}^{_{^{14}}}\right\rangle =f_{1}^{2}\overset{.}{x}%
^{2}-f_{4}^{2}\overset{.}{\alpha }^{2}<0.
\end{equation*}%
Therefore, we choose the following moving frame $e_{1},e_{2},e_{3},e_{4}$,
such that $e_{1},e_{2}$ are tangent to $S^{_{^{14}}}$ and $e_{3},e_{4}$ are
normal to $S^{_{^{14}}}.$ Also, we write as 
\begin{eqnarray*}
e_{1} &=&\frac{%
\begin{pmatrix}
f_{1}\overset{.}{x}\sinh x, \\ 
f_{4}\overset{.}{\alpha }\cosh \alpha , \\ 
f_{1}\overset{.}{x}\cosh x, \\ 
f_{4}\overset{.}{\alpha }\sinh \alpha 
\end{pmatrix}%
}{\sqrt{f_{4}^{2}\overset{.}{\alpha }^{2}-f_{1}^{2}\overset{.}{x}^{2}}}%
;e_{2}=\frac{%
\begin{pmatrix}
f_{1}^{\prime }\cosh x, \\ 
f_{4}^{\prime }\sinh \alpha , \\ 
f_{1}^{\prime }\sinh x, \\ 
f_{4}^{\prime }\cosh \alpha 
\end{pmatrix}%
}{\sqrt{-f_{1}^{\prime 2}+f_{4}^{\prime 2}}};\frac{%
\begin{pmatrix}
f_{4}\overset{.}{\alpha }\sinh x, \\ 
f_{1}\overset{.}{x}\cosh \alpha , \\ 
f_{4}\overset{.}{\alpha }\cosh x, \\ 
f_{1}\overset{.}{x}\sinh \alpha 
\end{pmatrix}%
}{\sqrt{f_{4}^{2}\overset{.}{\alpha }^{2}-f_{1}^{2}\overset{.}{x}^{2}}} \\
e_{4} &=&\frac{\left( f_{4}^{\prime }\cosh x,f_{1}^{\prime }\sinh \alpha
,f_{4}^{\prime }\sinh x,f_{1}^{\prime }\cosh \alpha \right) }{\sqrt{%
-f_{1}^{\prime 2}+f_{4}^{\prime 2}}}.
\end{eqnarray*}%
Then, we can easily get%
\begin{equation*}
\varepsilon _{1}=\left\langle e_{1},e_{1}\right\rangle =-1;\varepsilon
_{2}=\left\langle e_{2},e_{2}\right\rangle =1;\varepsilon _{3}=\left\langle
e_{3},e_{3}\right\rangle =-1;\varepsilon _{4}=\left\langle
e_{4},e_{4}\right\rangle =1.
\end{equation*}

By using (2.9), (2.10), (2.11), (2.12), we obtain the following coefficients
of the second fundamental form $h$ and the connection forms 
\begin{eqnarray*}
h_{11}^{3} &=&\frac{f_{1}f_{4}\left( \overset{..}{x}\overset{.}{\alpha }+%
\overset{.}{x}\overset{..}{\alpha }\right) }{\sqrt{f_{4}^{2}\overset{.}{%
\alpha }^{2}-f_{1}^{2}\overset{.}{x}^{2}}};h_{12}^{3}=\frac{\left(
f_{1}^{\prime }f_{4}^{{}}-f_{1}f_{4}^{\prime }\right) \overset{.}{x}\overset{%
.}{\alpha }}{\sqrt{f_{4}^{2}\overset{.}{\alpha }^{2}-f_{1}^{2}\overset{.}{x}%
^{2}}};h_{22}^{3}=0; \\
h_{11}^{4} &=&\frac{f_{1}^{\prime }f_{4}\overset{.}{\alpha }%
^{2}-f_{4}^{\prime }f_{1}\overset{.}{x}^{2}}{\sqrt{-f_{1}^{\prime
2}+f_{4}^{\prime 2}}};h_{22}^{4}=\frac{\left( f_{1}^{\prime }f_{4}^{\prime
\prime }-f_{1}^{\prime \prime }f_{4}^{\prime }\right) }{\sqrt{-f_{1}^{\prime
2}+f_{4}^{\prime 2}}};h_{12}^{4}=0;
\end{eqnarray*}%
and from (2.12) the mean curvature vector $H$ of the rotational surface $%
S^{_{^{14}}}$ is%
\begin{equation*}
H=\left\{ 
\begin{array}{c}
\frac{f_{1}f_{4}\left( \overset{..}{x}\overset{.}{\alpha }+\overset{.}{x}%
\overset{..}{\alpha }\right) }{2\sqrt{f_{4}^{2}\overset{.}{\alpha }%
^{2}-f_{1}^{2}\overset{.}{x}^{2}}} \\ 
+\frac{f_{4}^{\prime }f_{1}\overset{.}{x}^{2}-f_{1}^{\prime }f_{4}\overset{.}%
{\alpha }^{2}}{2\sqrt{-f_{1}^{\prime 2}+f_{4}^{\prime 2}}}%
\end{array}%
\right\} e_{3}+\left\{ \frac{\left( f_{1}^{\prime }f_{4}^{\prime \prime
}-f_{1}^{\prime \prime }f_{4}^{\prime }\right) }{2\sqrt{-f_{1}^{\prime
2}+f_{4}^{\prime 2}}}\right\} e_{4}.
\end{equation*}

The Gaussian curvature $K$ of the rotational surface $S^{_{^{14}}}$ is
obtained as%
\begin{equation*}
K=\underset{s=3}{\overset{4}{\sum }}\varepsilon _{s}\left[ h_{ij}^{s}\right]
=\frac{\left( 
\begin{array}{c}
f_{1}^{\prime }f_{4}^{{}} \\ 
-f_{1}f_{4}^{\prime }%
\end{array}%
\right) ^{2}\left( \overset{.}{x}\overset{.}{\alpha }\right) ^{2}}{f_{4}^{2}%
\overset{.}{\alpha }^{2}-f_{1}^{2}\overset{.}{x}^{2}}+\frac{\left( 
\begin{array}{c}
f_{1}^{\prime }f_{4}\overset{.}{\alpha }^{2} \\ 
-f_{4}^{\prime }f_{1}\overset{.}{x}^{2}%
\end{array}%
\right) \left( 
\begin{array}{c}
f_{1}^{\prime }f_{4}^{\prime \prime } \\ 
-f_{1}^{\prime \prime }f_{4}^{\prime }%
\end{array}%
\right) }{-f_{1}^{\prime 2}+f_{4}^{\prime 2}}.
\end{equation*}

Secondly, for the rotations $\Omega _{2}$ and $\Omega _{3}$, by using the
curve $\gamma (s)$, the surface of rotation $S_{23}$ around $\Pi _{\Omega
_{2}}(y).\Pi _{\Omega _{3}}(z)$ is given as follows%
\begin{equation}
S_{23}(y,z,s)=\Pi _{\Omega _{2}}(y).\Pi _{\Omega _{3}}(z).%
\begin{bmatrix}
f_{1}(s) \\ 
f_{2}(s) \\ 
f_{3}(s) \\ 
f_{4}(s)%
\end{bmatrix}%
=\left( 
\begin{array}{c}
f_{1}\cosh y+f_{4}\sinh y, \\ 
f_{2}\cosh z+f_{3}\sinh z, \\ 
f_{2}\sinh z+f_{3}\cosh z, \\ 
f_{1}\sinh y+f_{4}\cosh y%
\end{array}%
\right) ,  \tag{3.26a}
\end{equation}%
where $-\infty <z,y<\infty ,s\in I.$ Now, we consider the following the
surface of rotation%
\begin{equation}
S^{_{^{23}}}(y,z,s)=\left( f_{1}\cosh y,f_{2}\cosh z,f_{2}\sinh z,f_{1}\sinh
y\right) ,  \tag{3.26b}
\end{equation}%
where $f_{1}$ and $f_{2}$ are non-zero smooth functions and the curve $%
\gamma (s)=(f_{1}(s),f_{2}(s),0,0)$ lies on the $\xi \rho -$plane. For
(3.26b) we have the parametrizations 
\begin{eqnarray*}
S_{t}^{_{^{23}}}(y,z,s) &=&\left( f_{1}\overset{.}{y}\sinh y,f_{2}\overset{.}%
{z}\sinh z,f_{2}\overset{.}{z}\cosh z,f_{1}\overset{.}{y}\cosh y\right) ; \\
S_{tt}^{_{^{23}}}(y,z,s) &=&\left( 
\begin{array}{c}
f_{1}(\overset{..}{y}\sinh y+\overset{.}{y}^{2}\cosh y),f_{2}(\overset{..}{z}%
\sinh z+\overset{.}{z}^{2}\cosh z), \\ 
f_{2}(\overset{..}{z}\cosh z+\overset{.}{z}^{2}\sinh z),f_{1}(\overset{..}{y}%
\cosh y+\overset{.}{y}^{2}\sinh y)%
\end{array}%
\right) ; \\
S_{s}^{_{^{23}}}(y,z,s) &=&\left( f_{1}^{\prime }\cosh y,f_{2}^{\prime
}\cosh z,f_{2}^{\prime }\sinh z,f_{1}^{\prime }\sinh y\right)  \\
S_{ss}^{^{23}}(y,z,s) &=&\left( f_{1}^{\prime \prime }\cosh y,f_{2}^{\prime
\prime }\cosh z,f_{2}^{\prime \prime }\sinh z,f_{1}^{\prime \prime }\sinh
y\right) ; \\
S_{st}^{_{^{23}}}(y,z,s) &=&\left( f_{1}^{\prime }\overset{.}{y}\sinh
y,f_{2}^{\prime }\overset{.}{z}\sinh z,f_{2}^{\prime }\overset{.}{z}\cosh
z,f_{1}^{\prime }\overset{.}{y}\cosh y\right) 
\end{eqnarray*}%
and 
\begin{equation*}
\left\langle S_{s}^{^{23}},S_{s}^{_{^{23}}}\right\rangle =-f_{1}^{\prime
2}-f_{2}^{\prime 2}<0;\left\langle
S_{t}^{_{^{23}}},S_{t}^{_{^{23}}}\right\rangle =f_{2}^{2}\overset{.}{z}%
+f_{1}^{2}\overset{.}{y}>0.
\end{equation*}%
Hence, the following moving frame $e_{1},e_{2},e_{3},e_{4}$ can be chosen,
such that $e_{1},e_{2}$ are tangent to $S^{_{^{23}}}$ and $e_{3},$ $e_{4}$
are normal to $S^{_{^{23}}}$, we obtain as follows 
\begin{eqnarray*}
e_{1} &=&\frac{\left( f_{1}\overset{.}{y}\sinh y,f_{2}\overset{.}{z}\sinh
z,f_{2}\overset{.}{z}\cosh z,f_{1}\overset{.}{y}\cosh y\right) }{\sqrt{%
f_{2}^{2}\overset{.}{z}+f_{1}^{2}\overset{.}{y}}}; \\
e_{2} &=&\frac{\left( f_{1}^{\prime }\cosh y,f_{2}^{\prime }\cosh
z,f_{2}^{\prime }\sinh z,f_{1}^{\prime }\sinh y\right) }{\sqrt{f_{1}^{\prime
2}+f_{2}^{\prime 2}}}; \\
e_{3} &=&\frac{\left( f_{2}\overset{.}{z}\sinh y,f_{1}\overset{.}{y}\sinh
z,f_{1}\overset{.}{y}\cosh z,f_{2}\overset{.}{z}\cosh y\right) }{\sqrt{%
f_{2}^{2}\overset{.}{z}+f_{1}^{2}\overset{.}{y}}}; \\
e_{4} &=&\frac{\left( f_{2}^{\prime }\cosh y,f_{1}^{\prime }\cosh
z,f_{1}^{\prime }\sinh z,f_{2}^{\prime }\sinh y\right) }{\sqrt{f_{1}^{\prime
2}+f_{2}^{\prime 2}}}
\end{eqnarray*}%
Also, we have $\varepsilon _{1,3}=1;\varepsilon _{2,4}=-1.$ For the
equations (2.9), (2.10), (2.11), (2.12), the following coefficients of the
second fundamental form $h$ and the connection forms are obtained as 
\begin{eqnarray*}
h_{11}^{3} &=&\frac{f_{1}^{{}}f_{2}(\overset{.}{y}\overset{..}{z}+\overset{..%
}{y}\overset{.}{z})}{\sqrt{f_{2}^{2}\overset{.}{z}+f_{1}^{2}\overset{.}{y}}}%
;h_{12}^{3}=\frac{\left( f_{1}f_{2}^{\prime }+f_{1}^{\prime
}f_{2}^{{}}\right) \overset{.}{y}\overset{.}{z}}{\sqrt{f_{2}^{2}\overset{.}{z%
}+f_{1}^{2}\overset{.}{y}}};h_{22}^{3}=0; \\
h_{11}^{4} &=&\frac{-f_{1}^{{}}f_{2}^{\prime }\overset{.}{y}%
^{2}-f_{1}^{\prime }f_{2}\overset{.}{z}^{2}}{\sqrt{f_{1}^{\prime
2}+f_{2}^{\prime 2}}};h_{22}^{4}=\frac{-f_{1}^{\prime \prime }f_{2}^{\prime
}-f_{1}^{\prime }f_{2}^{\prime \prime }}{\sqrt{f_{1}^{\prime
2}+f_{2}^{\prime 2}}};h_{12}^{4}=0;
\end{eqnarray*}%
and from (2.12) the Gaussian curvature $K$ and the mean curvature vector $H$
of the rotational surface $S^{_{^{23}}}$ are obtained as follows%
\begin{eqnarray*}
H &=&\frac{f_{1}^{{}}f_{2}(\overset{.}{y}\overset{..}{z}+\overset{..}{y}%
\overset{.}{z})}{2\sqrt{f_{2}^{2}\overset{.}{z}+f_{1}^{2}\overset{.}{y}}}%
e_{3}+\frac{f_{1}^{{}}f_{2}^{\prime }\overset{.}{y}^{2}+f_{1}^{\prime }f_{2}%
\overset{.}{z}^{2}-f_{1}^{\prime \prime }f_{2}^{\prime }-f_{1}^{\prime
}f_{2}^{\prime \prime }}{2\sqrt{f_{1}^{\prime 2}+f_{2}^{\prime 2}}}e_{4} \\
K &=&-\frac{\left( f_{1}f_{2}^{\prime }+f_{1}^{\prime }f_{2}^{{}}\right)
^{2}\left( \overset{.}{y}\overset{.}{z}\right) ^{2}}{f_{2}^{2}\overset{.}{z}%
+f_{1}^{2}\overset{.}{y}}-\frac{\left( f_{1}^{{}}f_{2}^{\prime }\overset{.}{y%
}^{2}+f_{1}^{\prime }f_{2}\overset{.}{z}^{2}\right) \left( f_{1}^{\prime
\prime }f_{2}^{\prime }+f_{1}^{\prime }f_{2}^{\prime \prime }\right) }{%
f_{1}^{\prime 2}+f_{2}^{\prime 2}}.
\end{eqnarray*}

Also, $\Omega _{5}$ and $\Omega _{6}$ generate abelian sub-algebra being
elliptic. Therefore, we can write matrix $\Pi _{\Omega _{5}}(\beta )\Pi
_{\Omega _{6}}(\theta )$ being the rotational group of matrices. This
subgroup doesn't fix any axis and so it is not a rotation about any axis.
For the rotations $\Omega _{5}$ and $\Omega _{6},$ the matrices of rotations
of this surface can be written as $\Pi _{\Omega _{5}}(\beta )\Pi _{\Omega
_{6}}(\theta )$, by using a planar curve $\gamma $ with $s$ parameter the
surface of rotation $S_{56}$ around $\Pi _{\Omega _{5}}(\beta ).\Pi _{\Omega
_{6}}(\theta )$ can be parametrized as follows 
\begin{equation}
S_{56}(\beta ,\theta ,s)=\Pi _{\Omega _{5}}(\beta ).\Pi _{\Omega
_{6}}(\theta ).%
\begin{bmatrix}
f_{1}(s) \\ 
f_{2}(s) \\ 
f_{3}(s) \\ 
f_{4}(s)%
\end{bmatrix}%
=\left( 
\begin{array}{c}
f_{1}\cos \beta +f_{2}\sin \beta , \\ 
-f_{1}\sin \beta +f_{2}\cos \beta , \\ 
f_{3}\cos \theta +f_{4}\sin \theta , \\ 
-f_{3}\sin \theta +f_{4}\cos \theta 
\end{array}%
\right) ,  \tag{3.27a}
\end{equation}%
where $-\infty <\beta ,\theta <\infty ,s\in I.$ Now, we consider the
following rotational surface%
\begin{equation}
S^{_{^{56}}}(\beta \left( t\right) ,\theta \left( t\right) ,s)=\left(
f_{2}\sin \beta ,f_{2}\cos \beta ,f_{4}\sin \theta ,f_{4}\cos \theta \right)
,  \tag{3.27b}
\end{equation}%
where $f_{2}$ and $f_{4}$ are non-zero smooth functions and the curve $%
\gamma (s)=(0,f_{2}(s),0,f_{4}(s))$ lies on the $\rho \eta -$plane. From
(3.27b) we have the parametrizations 
\begin{eqnarray*}
S_{s}^{_{^{56}}}(\beta ,\theta ,s) &=&\left( f_{2}^{\prime }\sin \beta
,f_{2}^{\prime }\cos \beta ,f_{4}^{\prime }\sin \theta ,f_{4}^{\prime }\cos
\theta \right) ; \\
S_{ss}^{_{^{56}}}(\beta ,\theta ,s) &=&\left( f_{2}^{\prime \prime }\sin
\beta ,f_{2}^{\prime \prime }\cos \beta ,f_{4}^{\prime \prime }\sin \theta
,f_{4}^{\prime \prime }\cos \theta \right) ; \\
S_{st}^{_{^{56}}}(\beta ,\theta ,s) &=&\left( f_{2}^{\prime }\overset{.}{%
\beta }\cos \beta ,-f_{2}^{\prime }\overset{.}{\beta }\sin \beta
,f_{4}^{\prime }\overset{.}{\theta }\cos \theta ,-f_{4}^{\prime }\overset{.}{%
\theta }\sin \theta \right) ; \\
S_{t}^{_{^{56}}}(\beta ,\theta ,s) &=&\left( f_{2}\overset{.}{\beta }\cos
\beta ,-f_{2}\overset{.}{\beta }\sin \beta ,f_{4}\overset{.}{\theta }\cos
\theta ,-f_{4}\overset{.}{\theta }\sin \theta \right) ; \\
S_{tt}^{_{^{56}}}(\beta ,\theta ,s) &=&\left( 
\begin{array}{c}
f_{2}(\overset{..}{\beta }\cos \beta -\overset{.}{\beta }^{2}\sin \beta
),-f_{2}(\overset{..}{\beta }\sin \beta +\overset{.}{\beta }^{2}\cos \beta ),
\\ 
f_{4}(\overset{..}{\theta }\cos \theta -\overset{.}{\theta }^{2}\sin \theta
),-f_{4}(\overset{..}{\theta }\sin \theta +\overset{.}{\theta }^{2}\cos
\theta )%
\end{array}%
\right) 
\end{eqnarray*}%
and 
\begin{equation*}
\left\langle S_{s}^{_{^{56}}},S_{s}^{_{^{56}}}\right\rangle =-f_{2}^{\prime
2}+f_{4}^{\prime 2}>0;\left\langle S_{t}^{^{56}},S_{t}^{^{56}}\right\rangle
=-f_{2}^{2}\overset{.}{\beta }^{2}+f_{4}^{2}\overset{.}{\theta }<0.
\end{equation*}

For the following moving frame $e_{1},e_{2},e_{3},e_{4}$, we say that $%
e_{1},e_{2}$ are tangent to $S^{_{^{56}}}$ and $e_{3},e_{4}$ are normal to $%
S^{_{^{56}}}.$ Therefore, we get 
\begin{eqnarray*}
e_{1} &=&\frac{\left( f_{2}\overset{.}{\beta }\cos \beta ,-f_{2}\overset{.}{%
\beta }\sin \beta ,f_{4}\overset{.}{\theta }\cos \theta ,-f_{4}\overset{.}{%
\theta }\sin \theta \right) }{\sqrt{-f_{2}^{2}\overset{.}{\beta }%
^{2}+f_{4}^{2}\overset{.}{\theta }}}; \\
e_{2} &=&\frac{\left( f_{2}^{\prime }\sin \beta ,f_{2}^{\prime }\cos \beta
,f_{4}^{\prime }\sin \theta ,f_{4}^{\prime }\cos \theta \right) }{\sqrt{%
-f_{2}^{\prime 2}+f_{4}^{\prime 2}}} \\
e_{3} &=&\frac{\left( -f_{4}\overset{.}{\theta }\cos \beta ,f_{4}\overset{.}{%
\theta }\sin \beta ,-f_{2}\overset{.}{\beta }\cos \theta ,f_{2}\overset{.}{%
\beta }\sin \theta \right) }{\sqrt{-f_{2}^{2}\overset{.}{\beta }%
^{2}+f_{4}^{2}\overset{.}{\theta }}}; \\
e_{4} &=&\frac{\left( f_{4}^{\prime }\sin \beta ,f_{4}^{\prime }\cos \beta
,f_{2}^{\prime }\sin \theta ,f_{2}^{\prime }\cos \theta \right) }{\sqrt{%
-f_{2}^{\prime 2}+f_{4}^{\prime 2}}},
\end{eqnarray*}%
and we also get $\varepsilon _{1,4}=-1;\varepsilon _{2,3}=1.$ By considering
(2.9), (2.10), (2.11), (2.12), we obtain the following coefficients of the
second fundamental form $h$ and the connection forms 
\begin{eqnarray*}
h_{11}^{3} &=&\frac{f_{4}f_{2}\left( \overset{.}{\theta }\overset{..}{\beta }%
-\overset{.}{\beta }\overset{..}{\theta }\right) }{\sqrt{-f_{2}^{2}\overset{.%
}{\beta }^{2}+f_{4}^{2}\overset{.}{\theta }}};h_{12}^{3}=\frac{\left(
f_{2}^{\prime }f_{4}-f_{2}f_{4}^{\prime }\right) \overset{.}{\beta }\overset{%
.}{\theta }}{\sqrt{-f_{2}^{2}\overset{.}{\beta }^{2}+f_{4}^{2}\overset{.}{%
\theta }}};h_{22}^{3}=0; \\
h_{11}^{4} &=&\frac{f_{4}^{\prime }f_{2}\overset{.}{\beta }%
^{2}-f_{2}^{\prime }f_{4}\overset{.}{\theta }^{2}}{\sqrt{-f_{2}^{\prime
2}+f_{4}^{\prime 2}}};h_{22}^{3}=\frac{\left( -f_{2}^{\prime \prime
}f_{4}^{\prime }+f_{2}^{\prime }f_{4}^{\prime \prime }\right) }{\sqrt{%
-f_{2}^{\prime 2}+f_{4}^{\prime 2}}};h_{12}^{4}=0;
\end{eqnarray*}%
and from (2.12) the mean curvature vector $H$ of the rotation surface $%
S^{_{^{56}}}$ is%
\begin{equation*}
H=-\frac{f_{4}f_{2}\left( \overset{.}{\theta }\overset{..}{\beta }-\overset{.%
}{\beta }\overset{..}{\theta }\right) }{2\sqrt{-f_{2}^{2}\overset{.}{\beta }%
^{2}+f_{4}^{2}\overset{.}{\theta }}}e_{3}+\frac{\left( f_{4}^{\prime }f_{2}%
\overset{.}{\beta }^{2}-f_{2}^{\prime }f_{4}\overset{.}{\theta }%
^{2}+f_{2}^{\prime \prime }f_{4}^{\prime }-f_{2}^{\prime }f_{4}^{\prime
\prime }\right) }{2\sqrt{-f_{2}^{\prime 2}+f_{4}^{\prime 2}}}e_{4}
\end{equation*}%
and the Gaussian curvature $K$ of the rotation surface $S^{^{_{56}}}$ is
obtained as%
\begin{equation*}
K=-\frac{\left( f_{2}^{\prime }f_{4}-f_{2}f_{4}^{\prime }\right) ^{2}\left( 
\overset{.}{\beta }\overset{.}{\theta }\right) ^{2}}{-f_{2}^{2}\overset{.}{%
\beta }^{2}+f_{4}^{2}\overset{.}{\theta }}-\frac{\left( -f_{2}^{\prime
\prime }f_{4}^{\prime }+f_{2}^{\prime }f_{4}^{\prime \prime }\right)
(f_{4}^{\prime }f_{2}\overset{.}{\beta }^{2}-f_{2}^{\prime }f_{4}\overset{.}{%
\theta }^{2})^{2}}{-f_{2}^{\prime 2}+f_{4}^{\prime 2}}.
\end{equation*}
\end{proof}

\begin{example}
We consider the surfaces of rotation given as follows

\begin{enumerate}
\item For the curve $\gamma (s)=(s+\sinh s,0,0,s+\cosh s)$ the hyperbolic
surface of rotation is given as 
\begin{equation*}
S_{1}(x,\alpha ,s)=\left( 
\begin{array}{c}
(s+\sinh s)\cosh x,(s+\cosh s)\sinh \alpha , \\ 
(s+\sinh s)\sinh x,(s+\cosh s)\cosh \alpha 
\end{array}%
\right) .
\end{equation*}

\item For the curve $\gamma (s)=(s\cosh s,s\sinh s,0,0)$ the hyperbolic
surface of rotation is given as 
\begin{equation*}
S_{2}(y,z,s)=\left( s\cosh s\cosh y,s\sinh s\cosh z,s\sinh s\sinh z,s\cosh
s\sinh y\right) .
\end{equation*}

\item For the curve $\gamma (s)=(0,ax^{2}\sin s,0,ax^{2}\cos s)$ the
elliptic surface of rotation is given as 
\begin{equation*}
S_{3}(\beta ,\theta ,s)=\left( ax^{2}\sin s\sin \beta ,ax^{2}\sin s\cos
\beta ,ax^{2}\cos s\sin \theta ,ax^{2}\cos s\cos \theta \right) ;a,c\in 
\mathbb{R}
.
\end{equation*}
\end{enumerate}
\end{example}

\section{Conclusion}

In this paper, we gave different types of matrices of rotation which are the
subgroups of the manifold $M$ corresponding to rotation about a chosen axis
in $E^{4}$. Hence, we used two parameter matrices groups of rotations and we
gave the matrices of rotation corresponding to the appropriate subgroup of
the $E_{2}^{4}$ and we defined a brief description of rotational surfaces
using a curve and matrices in $E_{2}^{4}$. Furthermore, we examined the
special rotated surfaces generated by these matrices of rotation in $%
E_{2}^{4}$ and we expressed some certain results of describing the surface
obtaining Killing vector field in $E_{2}^{4}$ in detail. Also, we gave the
Gaussian curvature and the mean curvature of the surfaces of rotation.

The authors are currently working on the properties of these rotated
surfaces with a view to devising suitable metric in $E_{2}^{4}$ by adapting
the type of conservation laws considered in the paper. In our future
studies, we will study geodesics on the rotational surface obtained in $%
E_{2}^{4}$. Also, the physical terms such as specific energy and specific
angular momentum will be examined with the help of the conditions obtained
by using the Clairaut's theorem on these special surfaces.

\section{Acknowledgements}

The authors wish to express their thanks to the authors of literatures for
the supplied scientific aspects and idea for this study. 


\section{Funding}

Not applicable

\section{Conflicts of interest statement}

The authors have NO affiliations with or involvement in any organization or
entity with any financial interest or non-financial interest in the subject
matter or materials discussed in this manuscript.

\section{Declarations}

The authors declare that they have no known competing financial interests or
personal relationships that could have appeared to influence the work
reported in this paper.

\end{document}